\documentclass[a4paper,11pt]{amsart}
\usepackage[english]{babel}
\usepackage{amsmath,amssymb,amsthm,mathrsfs,amsopn,amscd}

\usepackage{verbatim,graphicx,color}
\usepackage[utf8]{inputenc}
\usepackage{indentfirst}
\usepackage[T1]{fontenc}
\usepackage{fancyhdr}

\newtheorem{thm}{Theorem}[section]
\newtheorem{lem}[thm]{Lemma}

\newtheorem{cor}[thm]{Corollary}

\def\ep{\varepsilon}
\def\SSN{{\mathbb S}^{N-1}}
\def\S1{{\mathbb S^1}}

\def\RN{\mathbb{R}^N}
\def\RR{\mathbb{R}}

\def\cH{\mathcal{H}}

\def\al{\alpha}
\def\la{\lambda}

\def\ka{\kappa}
\def\pa{\partial}

\def\de{\delta}

\renewcommand{\|}{|}

\newcommand{\cK}{\mathcal K}

\newcommand{\De}{\Delta}

\newcommand{\sums}{\sum\limits}

\newcommand{\si}{\sigma}
\newcommand{\Om}{\Omega}

\newcommand{\acc}{\`}

\newcommand{\RE}{\mathbb R}

\newcommand{\ovr}{\overline}

\newcommand{\ds}{\displaystyle}

\newcommand{\one}{\mathcal{X}}

\title[Characterization of ellipses]{Characterization of ellipses \\
as uniformly dense domains with respect to \\
a family of convex sets}

\author[Magnanini]{Rolando Magnanini}
\address{Dipartimento di Matematica ``U. Dini'', Universit\acc a di Firenze, viale Morgagni 67/A, 50134 Firenze, Italy}
\email{magnanin@math.unifi.it}

\author[Marini]{Michele Marini}
\address{Scuola Normale Superiore, Piazza dei Cavalieri 7, 56126 Pisa, Italy}
\email{michele.marini@sns.it}


\begin{document}

\begin{abstract}
Given $K\subset\RN$ a convex body containing the origin, 
a measurable set $G\subset\RN$ with positive Lebesgue measure is
said to be uniformly $K$-dense if the measure of the
sets $G\cap (x+r\,K)$ is constant on the boundary of $G$ for
any fixed $r>0.$ For $N=2,$ we prove that $G$ is uniformly $K$-dense if 
and only if $K$ and $G$ are homothetic ellipses. Our result improves one
obtained by Amar, Berrone and Gianni in two respects:
it removes the regularity assumptions on $K$ and $G;$ 
by using Minkowski's inequality and an affine inequality,
in the proof it is not necessary to compute higher-order terms in the Taylor expansion near $r=0$ for
the measure of $G\cap (x+r\,K).$ 
\end{abstract}

\maketitle

\section{Introduction}

Let $K$ be a convex body containing the origin of $\RN$ and $G$ be a measurable subset of $\RN$  
with positive Lebesgue measure $V(G).$ 
For each fixed $r>0,$ we define a density function
$\de_K:\RN\times(0,\infty)\rightarrow \RE$ as follows:
\begin{equation}
\label{cdense}
\de_K(x,r)=\frac{V(G\cap (x+r\,K))}{V(r\,K)}, \ x\in\RN.
\end{equation}
Here, $x+r\,K$ denotes the translation by a vector $x$ of
a dilation of $K$ by a factor $r>0$.  
\par
We say that $G$ is {\it uniformly $K$-dense,} or simply {\it $K$-dense} for short, if 
there is a function $c:(0,\infty)\to(0,\infty)$ such that 
$$
\de_K(x,r)=c(r) \ \mbox{ for every } \ (x,r)\in\pa G\times(0,\infty),
$$ 
where $\pa G$ denotes the topological boundary of the set $G$.
\par
When $K$ is the unit ball $B$ of $\RN,$ $K$-dense domains have been studied in
\cite{MPS} in connection with the so-called {\it stationary isothermic surfaces} --- 
the time-invariant level surfaces of solutions of the heat equation. There,
it is proved that a domain $G$ is uniformly dense ($B$-dense in our terminology)
if and only if the solution $U=U(x,t)$
of the following Cauchy problem
\begin{equation}
\label{cauchy}
U_t=\De U \ \mbox{ in } \ \RN\times(0,\infty), \ \ U=\one_G  \ \mbox{ on } \ \RN\times\{0\},
\end{equation}
is such that
$$
U(x,t)=a(t) \ \mbox{ for } \ (x,t)\in\pa G\times(0,\infty),
$$
for some function $a:(0,\infty)\to(0,1)$ (here, $\one_G$ denotes the 
characteristic function of the set $G$). The latter condition qualifies
$\pa G$ as a stationary isothermic surface for $U.$ 
The aforementioned equivalence easily follows from the fact that the solution of \eqref{cauchy}
can be written as
$$
U(x,t)=\frac{|B|}{\pi^{N/2}\sqrt{t}}\,\int_0^\infty \de_B(x,\sqrt{4t}\,\si)\,\si^N\,e^{-\si^2}\,d\si.
$$
\par
The study of stationary isothermic surfaces was motivated by a problem posed
by M.S. Klamkin in \cite{Kl}. Contributions in that field can be found in
\cite{Al1}-\cite{Al2}, \cite{MS1}-\cite{MS2} for the case of the initial-Dirichlet 
boundary value problem, \cite{Sa} for the initial-Neumann 
boundary value problem and in \cite{MS3}-\cite{MS5} for some
generalizations to nonlinear problems.  
\par
Problem \eqref{cauchy} is the simplest setting in which stationary isothermic surfaces have been
considered and their equivalence with $B$-dense domains have been 
instrumental to obtain an almost complete characterization for them. In fact, in \cite{MPS},
it is shown that if $u$ is the solution of \eqref{cauchy} and $\pa G$ is connected, bounded and stationary for $u$, 
then $\pa G$ is a sphere; if $\pa G$ is 
connected, unbounded and stationary, then  
it is a straight line, if $N=2;$ it is either a spherical cylinder or a minimal surface 
(which reduces to a plane, if its total curvature is finite), if $N=3;$ 
its principal curvatures must satisfy certain necessary constraints for $N\ge 4;$
it is also shown that the right helicoid is a stationary isothermic surface with infinite
total curvature. Finally, it is observed in \cite{MPS} that, if $E$ is an {\it ellipsoid,}
then $E$-dense domains are obtained as affine images of $B$-dense ones;
in particular, any bounded $E$-dense domain must be homothetic to $E,$
and hence an ellipsoid itself.
\par
The case of general $K$-dense domains have been considered by
Amar, Berrone and Gianni in \cite{ABG}, when $N=2.$
There, by calculating, for a fixed $x\in G,$ 
the Taylor expansion of the function $\de_K$ in \eqref{cdense} as $r\to 0^+$ up to the third order,
it is proved that, if $\pa G$ is $C^4$-smooth, $\pa K$ is $C^2$-smooth and $G$ is $K$-dense,
then both $G$ and $K$ must be homothetic to an ellipse $E.$ 
It is reasonable to conjecture that this conclusion still holds when $N\ge 3,$
that is 
\begin{quotation}
\it $G$ is $K$-dense if and ony if $K$ and $G$ are homothetic ellipsoids. 
\end{quotation}
Nevertheless, as we shall explain below, it is seems difficult to extend the analysis employed in \cite{ABG} to the case $N\ge 3:$
other means must be developed. The purpose of this paper is to investigate in that direction. 
\par
A geometrical analysis of the the computations made in \cite{ABG}
gives some useful information: (i) the first relevant coefficient 
in the Taylor expansion for $\de_K$ is related to the volume
of certain subsets of $K$ and can be used to give information on its symmetry; 
(ii) the second one is somewhat related to a weighted curvature of $\pa G$ at $x$; 
(iii) in the third one, the derivatives (up to the order $2$) of the curvature appear.
It is reasonable to expect that the higher-order coefficients contain information
about higher-order derivatives of the curvature of $\pa G$.
We shall see that, in general dimension, it is relatively easy to compute 
the first and second coefficient and it will be clear that is very difficult to compute the 
higher-order ones. In any case, higher-order terms can only give local
information about the surface $\pa G;$ thus, to have hope to prove the
conjectured result, we must use some global information.
\par 
The main result of this paper is an improvement of Amar, Berrone and Gianni's result.
\begin{thm}
\label{th:ellipse}
Let $K\subset\RE^2$ be a convex body and let $G$ be a bounded measurable set in $\RR^2.$ 
\par
If $G$ is $K$-dense, then $K$ and $G$ are ellipses that differ from one another
by a homothety.
\end{thm}
\par
The improvements we introduce are mainly two: we remove the regularity assumptions on $K$ and $G;$
our proof only relies on items (i) and (ii), Minkowski's inequality for mixed volumes and a variant of the affine isoperimetric inequality: we 
thus avoid the use of the higher-order local information mentioned in (iii).
In the remainder of this section, we explain in detail the main steps of our argument.
\par
We begin by showing, in general dimension, that a $K$-dense domain $G$ is necessarily strictly convex and,
{\em no matter how regular $K$ is,} at least $C^{1,1}$-smooth, that is 
its boundary is locally the graph of a differentiable function with 
Lipschitz continuous derivatives (see Theorems \ref{convexity} and \ref{th:regomega}). These two properties are proved by showing
that $G$ is a level set for a regular value of a $C^{1,1}$-smooth convex function. 
\par
Then, we continue our investigation and observe that the requirement that $G$ is $K$-dense does not only imply the regularity of $G$
but also that of $K$ itself, and more: in fact, under the additional assumption that $K$ is 
centrally symmetric, we show that $K$ must be $C^{1,1}$-smooth, strictly convex and
that $K=G-G$ (i.e. $K$ is {\it Minkowski sum} of $G$ and $-G$) up to homotheties (Theorem \ref{th:minksum}). 
\par
A first by-product of this result is that the gain on the regularity of $K$ implies a gain in that of $G,$
that must be $C^{2,1}$-smooth. A second consequence pertains the case $N=2:$ since we are able to prove in this case that 
the $K$-density of $G$ implies the central symmetry of $K$ and $G,$ 
we obtain that $K$ and $G$ only differ by a homothety and are both strictly convex and $C^\infty$-smooth.
\par
However, the very importance of Theorem \ref{th:minksum} is that
it points towards the direction of the desired conjecture,
in the sense that, with the additional assumption that also $G$ be centrally symmetric,
we obtain that $G$ is a dilate of $K$ --- as predicted by the conjecture --- and moreover, 
by a bootstrap argument, we find that
both $K$ and $G$ must be $C^\infty$-smooth.  
\par
The next step of our argument is the computation of the first and second coefficient in the Taylor expansion
for $\de_K(x,r)$ in general dimension. Differently from what was done in \cite{ABG}, 
we privilege a geometrical point of view; in fact, we obtain the following formula:
\begin{equation}
\label{taylor} 
\de_K(x,r)=\de_0(x)-\de_1(x)\,r+o(r) \ \mbox{ as } \ r\to 0^+,
\end{equation}
where
\begin{equation}
\label{coeff0}
\de_0(x)=\frac{V(K\cap H^+_{\nu(x)})}{V(K)}, \ \ x\in\pa G,
\end{equation}
and 
\begin{equation}
\label{coeff1}
\de_1(x)=\frac1{2\,V(K)}\,\sums_{i=1}^{N-1} m_i(x)\,\ka_i(x), \ \ x\in\pa G.
\end{equation}
Here, 
\begin{equation}
\label{defmi}
m_i(x)=\int_{K\cap\pi_{\nu(x)}} \langle\xi, e_i(x)\rangle^2\,d\cH^{N-1}_\xi, \ i=1,\dots, N-1;
\end{equation}
$\nu(x)$ denotes the inward unit normal to $\pa G$ at $x;$ for any $u\in\SSN,$ 
$H^+_u$ and $\pi_u$
are respectively the half-space $\{y\in\RN: \langle y\cdot u\rangle\ge 0\}$  and the hyperplane $\pa H^+_u;$ 
$\ka_1(x),\dots, \ka_{N-1}(x)$ and $e_1(x),\dots, e_{N-1}(x)$ are respectively the principal curvatures and
directions of $\pa G$ at $x;$ 
$\cH^{N-1}_\xi$ is the $(N-1)$-dimensional Hausdorff measure. 
\par
When $G$ is $K$-dense, easy consequences of \eqref{taylor}, \eqref{coeff0} and \eqref{coeff1} are:
\begin{equation}
\label{cond1}
V(K\cap H^+_{\nu(x)})=\frac12\,V(K) \ \ x\in\pa G,
\end{equation}
and
\begin{equation}
\label{cond2}
\sums_{i=1}^{N-1} m_i(x)\,\ka_i(x)= c\,V(K), \ \ x\in\pa G,
\end{equation}
where $c$ is a constant. Condition \eqref{cond1} 
gives some sort of symmetry for $K$ (that, for $N=2,$ implies its central symmetry, 
as already observed in \cite{ABG}). Condition \eqref{cond2} 
is a constraint between the curvatures of $\pa G$ and certain {\it moments
of inertia} of the central sections of $K.$ When $N=2,$ it means that
the curvature of $\pa G$ and the radial function of $K$ must be
somewhat related. This last information is crucial since it
implies that $K$ and $G$ must be homothetic and, with the help
of Minkowski's inequality and
an affine inequality, that both
must be ellipses.

\setcounter{equation}{0}

\section{Convexity and regularity of $K$-dense domains.}

 Let $\cK_0^N$ be the set of convex bodies of $\RN$ that contain the origin in their interior; 
for $K\in\cK_0^N$ let $\|\cdot\|_K :\RN\rightarrow \mathbb R^+$ denote the {\em gauge} of the set $K$, that is
\[
\|x\|_K=\min\{r>0 : x\in rK\}.
\]
It is well-known that $x+r\,K=\{ y\in\RN: \|y-x\|_K\le r\};$ since, when $K$ is symmetric with respect to the origin, $\|\cdot\|_K$ is a norm, 
$B_K(x,r)$ is a convenient notation for the set $x+rK$. When $K=B,$ $\|\cdot\|_B$ is the euclidean norm
and we shall drop the subscript $B.$
\par
Given a measure $\mu$ on $\mathbb R^+$ and set $\phi(t)=\mu([0,t))$, we define a function $f^\phi:\RN\rightarrow\mathbb R$ as follows:
\begin{equation}
\label{definf}
        f^\phi(x)=\int_G \phi(\|y-x\|_K)\,dy=\int_G \phi(\|x-y\|_{-K})\,dy;
\end{equation}
$f^\phi$ is thus the convolution of the characteristic function $\one_G$ and
the composition of $\phi$ with the gauge of $-K.$

If $\mu$ is a Borel and locally finite measure, we can use the {\em layer-cake} representation theorem 
(see \cite{LL} for instance) in order to write:
\begin{equation}
\label{layer}
f^\phi (x)=\int_0^{+\infty} V( G\cap\{y :\|y-x\|_K>t\})\, d\mu=\int_0^{+\infty} V( G\setminus B_K(x,t))\, d\mu.
\end{equation}
\par
If $ G$ is $K$-dense, the last integral does not depend on $x$, for $x\in\partial G$.
Conversely, if $f^\phi(x)$ is constant on $\partial G$ for every chioice of the measure $\mu$, 
for each given $r>0$ we can set $\mu=\delta_r$ (the Dirac's delta measure centered at $r$) in \eqref{layer}
and obtain that $f^\phi(x)=V( G\setminus B_K(x,r)).$
When $G$ has finite measure, the assumption on $f^\phi$ and the 
fact that $r$ is arbitrary imply that $G$ must be $K$-dense. 
Thus, we can state the following characterization.

\begin{thm}
\label{characterisation}
Let $G$ be a bounded\footnote{It is possible to replace this assumption by asking that $V( G)<\infty;$ 
however, it turns out that there not exists any unbounded $K$-dense set of finite measure.}, measurable subset of $\RN$ with $V(G)>0$.
Then the following conditions are equivalent:
\begin{itemize}
\item[(i)] $ G$ is $K$-dense,
\item[(ii)] for every Borel, locally finite measure $\mu$ on $\mathbb R^+$, the function $f^\phi$ defined in \eqref{definf} does not depend on $x$, for $x\in\pa G$.
\end{itemize}
\end{thm}
The following lemma is instrumental to prove the convexity of $G;$ its proof is straightforward. 
\begin{lem}
Let the function $\phi(t)=\mu([0,t))$ be convex, increasing and non-constant,  
and let $f^\phi$ be the function defined in \eqref{definf}.
Then:
\begin{itemize}
\item[(i)] $f^\phi$ is convex and hence, in particular, continuous;
\item[(ii)] $f^\phi$ is coercive, that is $f^\phi\to+\infty$ as $|x|\to\infty.$ 
\end{itemize}
\end{lem}

%

\begin{thm}
\label{convexity}
Let $G$ be a bounded $K$-dense set; then $ G$ is strictly convex.
\par
Moreover, if the function $\phi(t)=\mu([0,t))$ is convex and strictly increasing, 
then $G$ is a regular level set for $f^\phi$.
\end{thm}

\begin{proof}
First, we show that, if $\phi$ satisfies the assumptions, then $f^\phi$ cannot be constant 
on a segment whose middle point belongs to $\overline{G}$.
\par 
By contradiction, let $x$ and $y$ be the endpoints of a segment on which $f^\phi$ is constant and suppose the midpoint $\frac{1}{2}(x+y)\in\overline G$; then
\[
\int_{G}\left\{ \phi(\|z-x\|_K)/2+ \phi(\|z-y\|_K)/2-
\phi\left(\left\|z-(x+y)/2\right\|_K\right)\right\}\, dz=0.
\]
Since the integrand is always non-negative, we get that
\[
2\,\phi\left(\left\|z-(x+y)/2\right\|_K\right)=\phi(\|z-x\|_K)+\phi(\|z-y\|_K)
\]
for every $z\in\overline G,$ since both $\phi$ and $\|\cdot\|_K$ are continuous\footnote{This is clear when $G$ is connected. Otherwise, it is sufficient that, for each $x\in\overline G$, every neighborhood of $x$ has intersection with $G$ of positive measure. This is guaranteed by the fact that $G$ is $K$-dense.}.
Thus, if we choose $z=\frac{1}{2}(x+y)$ we get a contradiction.
\par 
Therefore, we can claim that the function $f^\phi$ does not reach its minimum on the boundary of $ G$, otherwise $f^\phi$ would be constant on the convex hull of $\partial G$ which contains a segment whose middle point belongs to $\overline G$.
Indeed consider a line, say $r$, containing at least three points of $G$, say $x$, $y$ and $z$, with $y\in ]xz[$\footnote{We denote by $]xz[$ the relatively open segment from $x$ to $z$.}; then, being $G$ bounded, $\pa G$ intersects every connected component $r\setminus ]xz[$ and thus every point of $]xz[$ belongs to the convex hull of $\pa G$; simply choose a segment contained in $]xz[$ whose middle point is $y$.\\
Hence, there exists a positive number $s$ such that 
the set $A$ where $f^\phi<s$ is open, bounded and convex;
also, $\pa G\subseteq\pa A=\{ x\in\RN: f^\phi(x)=s\}.$ 
It is now easy to check that this property implies that $A\subseteq  G\subseteq\overline A$ 
and, in particular, that $G$ is convex and hence stictly convex.
\end{proof}

\begin{cor}\label{maxdist}
Let $G$ be a $K$-dense body; then the function
\[
x\mapsto\max_{y\in G}\|y-x\|_K
\]
is constant on $\pa G$.
\end{cor}

\begin{proof}
Let $x$ and $z\in\pa G$ and suppose by contradiction that 
$$
d_1=\max_{y\in G}\|y-x\|_K<\max_{y\in G}\|y-z\|_K=d_2.
$$
Then $ G\setminus B_K(z,d_1)\neq\varnothing$ and hence $V(G\setminus B_K(z,d_1))>0,$
being $G$ a body and $B_K(z,d_1)$ open; thus, 
\begin{eqnarray*}
&V(G\cap B_K(x,d_1))=V(G)=\\
&V(G\setminus B_K(z,d_1))+V(G\cap B_K(z,d_1))>V(G\cap B_K(z,d_1)).
\end{eqnarray*}
\end{proof}

We now study the regularity of $K$-dense sets.

\begin{thm}
\label{th:regomega}
Let $G$ be a $K$-dense body; then $\pa G$ is of class $C^{1,1},$
that is $\pa G$ is locally the graph of a $C^{1,1}$-smooth function.
\end{thm}

\begin{proof}
Set $f=f^\phi$ with $\phi(t)=t.$ 
By Theorem \ref{convexity}, it is sufficient to show that $f\in C^{1,1}$.
\par
Consider the incremental ratio of $f$ at $x$ in a canonical direction $e_i:$ 
\[
\frac{f(x+te_i)-f(x)}{t}=\int_ G\frac{\|x-z+te_i\|_{-K}-\|x-z\|_{-K}}{t}\,dz.
\]
Since $\|\cdot\|_{-K}$ is almost everywhere differentiable and its gradient is a bounded map over $\RN$, 
by the dominated convergence theorem, we obtain that the partial derivative $\pa_{x_i}f(x)$
exists and equals
\begin{equation*}
\label{partial}
\int_G \frac{\pa}{\pa x_i} \| x-z\|_{-K}\,dz=\int_{\RN} \one_G(x-z)\,\frac{\pa}{\pa z_i} \| z\|_{-K}\,dz,
\end{equation*}
and the second factor in the integrand is bounded almost everywhere by a constant, say, $L.$
Thus, for $x, y\in\RN,$ we obtain the estimate:
\begin{eqnarray*}
|\pa_{x_i}f(x)-\pa_{x_i}f(y)|&\le& L\,\int_{\RN}|\one_ G(x-z)-\one_ G(y-z)|\,dz\le \\
&&L\,P(\pa G)\,\|x-y\|,
\end{eqnarray*}
since $G$ is convex and bounded (here, $P(\pa G)$ denotes the perimeter of $G$).
\par
Therefore, $f$ is differentiable and has Lipschitz continuous partial derivatives.
\end{proof}

Since the function $\|\cdot\|_K$ has the same regularity as $\pa K$ 
at all points of $\RN$ except the origin, then if $\pa K\in C^{m,1}$ for some integer $m,$ 
by the same arguments used in the proof of Theorem \ref{th:regomega}, we can
easily prove the following result.

\begin{thm}
\label{th:induc}
Let $G$ be a bounded $K$-dense set, and let $\pa K\in C^{m,1}$ for some integer $m.$ 
Then $\pa G\in C^{m+1,1}.$
\end{thm}

\begin{cor}\label{bootstrap}
Let $G$ be a bounded $K$-dense set. If the class of homothetical images of $K$ contains $G$, 
then $\pa G\in C^\infty$.
\end{cor}

\begin{proof}
We show that $G\in\mathcal C^{m,1}$  for every $m\in\mathbb N$ by induction on $m$.
The base step is exhibited Theorem \ref{th:regomega}; the inductive step is the matter of Theorem \ref{th:induc}.
\end{proof}

The following result shows that, surprisingly, at least when
$K$ is centrally symmetric, the existence of a $K$-dense set implies
some regularity of $K$ itself.

\begin{thm}
\label{th:minksum}

Let $K$ be a convex body symmetric with respect to the origin of $\RN$, and let $G$ be a $K$-dense body. 
Then it holds that

\begin{itemize}
\item[(a)] $K= G-G$, up to homotheties;
\item[(b)] $K$ is strictly convex;
\item[(c)] $\pa K$ and $\pa G$ are respectively $C^{1,1}$-smooth and $C^{2,1}$-smooth.
\end{itemize}
\end{thm}

\begin{proof}
Recall that, since $K$ is convex, to each point $x\in\pa K$ we can associate its (non-empty) 
{\em normal cone} $N_K(x),$ which is the set of vectors $w$ such that 
$\langle y-x, w\rangle\ge 0$ for every $y\in K.$ Thus, in order to prove the 
differentiability of $\pa K,$ we only need to prove that $N_K(x)\cap\SSN$ 
contains only one vector for every $x\in\pa K.$ 
\par 
(a) Without loss of generality, let us suppose that 
$$
\max_{y\in G}\|y-x\|_K=1  \ \mbox{ for every } \ x\in\pa G.
$$
We have that
$$
\max_{y\in G-x}\|y\|_K=1,
$$ and hence $G-x\subseteq K$ for every $x\in\pa G.$
It follows that $ G- G\subseteq K.$ Indeed, if $z\in G-G,$ then $z=x-y$ for some points $x,y\in G;$
since $G$ is convex, there are points $x_1$ and $x_2$ in $\pa G$ and a number $0\leq\lambda\leq1$ such that
$x=\lambda x_1+(1-\lambda)x_2.$
Hence,
\[
z=\la\,(y-x_1)+(1-\la)(y-x_2).
\]
Since $K$ is convex and contains both $y-x_1$ and $y-x_2,$ we get that $z\in K.$
\par 
Viceversa, let $x$ be an exposed point of $\pa K$ and let $u\in\mathbb S^{N-1}$ 
be such that $H_u$ is the supporting hyperplane which intersects $K$ only at the point $x$.
\par
Next, choose $y\in\pa G$ such that the (inward) unit normal to $\pa G$ at $y$, $\nu_ G(y)$, coincides with $-u$
(it exists since we already know that $G$ is smooth and strictly convex).
Also, pick a point $z\in\pa G$ that maximizes the $K$-distance from $y$, 
that is, such that $\|y-z\|_K=1$. 
Note that $y-z\in (G-z)\cap\pa K$ and, since $ G-z\subseteq K$, 
we get the following reverse inclusion for the normal cones:
\[
N_K(y-z)\cap\SSN\subseteq\{-\nu_{ G-z}(y-z)\}=\{-\nu_{ G}(y)\}=\{ u\}.
\]
Hence, our choice of $x$ and $u$ allows us to write $x=y-z$.
Thus, $G-G$ contains all the exposed points of $\pa K$ and hence, being $K$ a closed convex set, it 
must contain also $K.$ 
\par
(b) It easily follows from (a) and Theorem \ref{convexity}.
\par
(c) From (a) and Theorem \ref{th:regomega}, it follows that $\pa K$ is $C^{1,1}$-smooth,
since the Minkowski sum of $C^{1,1}$ sets is $C^{1,1}$\footnote{See for instance \cite{KP}.}.
Theorem \ref{th:induc} then implies that $\pa G$ is $C^{2,1}$-smooth.
\end{proof}

\begin{cor}\label{cor:omegasim}
If, in addition to the assumptions of Theorem \ref{th:minksum}, $G$ is centrally symmetric, 
then $G=K$ (up to homotheties) and $\pa G$ (and $\pa K$) is $C^\infty$-smooth.
\end{cor}

\setcounter{equation}{0}

\section{Asymptotics as $r\to 0^+.$}

Consistently with what defined in Section 1, given a unit vector $u\in\SSN,$ 
we write $H_u^+=\{x\in\RN : \langle x, u\rangle \geq 0\}$ and $H_u^-=H_{-u}^+.$
Also, since our focus is on $K$-dense sets, without loss of generality, we can always suppose that
$G$ is convex.

\begin{thm}\label{1order}
Let $ G$ and $K$ be convex bodies and suppose that $\pa G$ is differentiable at $x.$ 
Then
\[
\lim_{r\to 0^+}\de_K(x,r)=\frac{V(K\cap H_{\nu(x)}^+)}{V(K)}.
\]
\par 
In particular, if $G$ is $K$-dense, then 
\begin{equation}
\label{halfC}
V(K\cap H_u^+)=\frac12\,V(K) \quad\mbox{for all $u\in\mathbb S^{N-1}.$}
\end{equation}

\end{thm}

\begin{proof}
For $r>0$ we have:
\begin{equation}\label{sviluppo1}
    r^{-N}V( G\cap (x+r\,K))=V\left(\frac{G -x}{r}\cap K\right).
\end{equation}

Since $\pa G$ is differentiable at $x$, as $r$ decreases to $0,$ $\ds\frac{G -x}{r}\cap K$ 
increases to $H_{\nu(x)}^+\cap K.$ The first claim of the theorem then follows from the monotone
convergence theorem.
\par
Now, suppose that $G$ is $K$-dense. Then, the Gauss map from $\pa G$ to $\SSN$ 
that takes any $x\in\pa G$ to the outward normal unit vector $\nu(x)$ 
is surjective. Hence, for every $u\in\SSN,$ 
there exist $x,x'\in\pa G$ such that $u=\nu(x)=-\nu(x')$.
\par
Since $G$ is $K$-dense, then the quantity $V(K\cap H_{\nu(x)}^+)$ does not depend on $x$, for $x\in\pa G.$
Thus, our choice of $x$ and $x'$ enables us to write that
\[
V(K\cap H_{\nu(x)}^+)=V(K\cap H_{\nu(x')}^+)=V(K\cap H_{\nu(x)}^-).
\]
Since $V(K\cap H_{\nu(x)}^-)+V(K\cap H_{\nu(x)}^+)=V(K),$ 
then we find that 
\[
V(K\cap H_{\nu(x)}^+)=\frac12\,V(K).
\]
\end{proof}

\begin{cor}
\label{cor:2dreg}
If $G$ is $K$-dense, then
\begin{equation}
\label{centroids}
\int_{K\cap\pi_u}\langle y,w\rangle\,dy=0 \ \mbox{ for every } \ u, w\in\SSN \ \mbox{ with } \ \langle u,w\rangle=0. 
\end{equation}
In particular, when $N=2,$ $K$ is centrally symmetric.
\end{cor}

\begin{proof}
Let $u$ and $v\in\SSN$, then:
\begin{eqnarray*}
&&V(K\cap H_{v}^+\cap H_u^+)+V(K\cap H_{v}^+\cap H_u^-)=V(K\cap H_{v}^+)=\\
&&V(K\cap H_{v}^-)=V(K\cap H_{v}^-\cap H_u^+)+V(K\cap H_{v}^-\cap H_u^-),
\end{eqnarray*}
and also
\begin{eqnarray*}
&&V(K\cap H_{v}^+\cap H_u^+)+V(K\cap H_{v}^-\cap H_u^+)=V(K\cap H_u^+)=\\
&&V(K\cap H_u^-)=V(K\cap H_{v}^+\cap H_u^-)+V(K\cap H_{v}^-\cap H_u^-).
\end{eqnarray*}
Thus,
\begin{equation}
\label{wedge}
V(K\cap H_{v}^+\cap H_u^+)=V(K\cap H_{v}^-\cap H_u^-).
\end{equation}
\par
Now fix $\ep>0,$ a unit vector $u$  and choose $v=-u\,\cos\ep+w\,\sin\ep,$ where $w$ is
a unit vector orthogonal to $u;$ we can write that
$$
K\cap H_{v}^+\cap H_u^+=\{ y+t u: \langle y,u\rangle=0, \|y+t u\|_K\le 1, 0\le t\le\langle y,w\rangle\tan\ep\}
$$ 
and, by a re-scaling in the variable $t,$ we get that
\begin{eqnarray*}
&\displaystyle\frac1{\ep}\,V(K\cap H_{v}^+\cap H_u^+)=\\
&\displaystyle\frac{\tan\ep}{\ep}\,V(\{ y+\tau u: \langle y,u\rangle=0, \|y+\tau u \tan\ep\|_K\le 1, 0\le\tau\le\langle y,w\rangle\}).
\end{eqnarray*}
As $\ep\to 0,$ $v\to u$ and we can easily infer that
\begin{eqnarray*}
&&\lim_{\ep\to 0}\frac1{\ep}\,V(H_{v}^+\cap H_u^+\cap K)=\\
&&\qquad\qquad V(\{ y+\tau u: \langle y,u\rangle=0, \|y\|_K\le 1, 0\le\tau\le\langle y,w\rangle\})=\\
&&\qquad\qquad\int_{K\cap\pi_u\cap H^+_w}\langle y,w\rangle\,dy.
\end{eqnarray*}
By the same argument, we obtain that 
\begin{eqnarray*}
\lim_{\ep\to 0}\frac1{\ep}\,V(H_{v}^-\cap H_u^-\cap K)=
-\int_{K\cap\pi_u\cap H^-_w}\langle y,w\rangle\,dy
\end{eqnarray*}
and hence \eqref{wedge} implies \eqref{centroids}.
\par
By a simple inspection of the proof, in the case $N=2$ we
easily obtain that $-K=K.$ 
\end{proof}

Theorem \ref{th:minksum} and Corollary \ref{cor:2dreg} immediately 
imply the following result.

\begin{cor}\label{2dreg}
Let $ G\subset\mathbb R^2$ be a $K$-dense body, then $ G\in C^{2,1}$.
\end{cor}

We now compute the second term in the asymptotic expansion for $\de_K.$

\begin{thm}
Let $ G\subset\mathbb R^N$ be a convex body with $C^2$-smooth boundary, let $x\in\pa G$
and denote by  $\ka_1(x),\dots, \ka_{N-1}(x)$ the principal curvatures of $\pa G$ at $x$ with respect to
the inward normal unit vector. 
\par
Then, we have the formula:
\begin{equation}
\label{curvatures}
    \lim_{r\to 0^+}\frac{\de_K(x,r)-\de_0(x)}{r}=
    -\frac{1}{2\,V(K)}\,\sums_{i=1}^{N-1}m_i(x)\,\ka_i(x),
\end{equation}
where the coefficients $m_i(x)$ are given by \eqref{defmi}.
Therefore, \eqref{taylor} holds.
\end{thm}

\begin{proof}
We choose a coordinate system $\{e_1,\ldots,e_{N-1}, \nu\}$ around the point $x\in\partial G$ such that $e_i$, for $i=1,\ldots,N-1$, is the $i$-th principal direction of $\pa G$ at $x$ and $\nu=\nu(x)$ is the normal.
\par
In these coordinates $B_K(x,r)$ can be written as
\[
B_K(x,r)=\left\{x+\sum_{i=1}^{N-1}z_i e_i +z_N \nu: z\in\RN,\ \left\|\sum_{i=1}^{N-1}z_i e_i +z_N \nu\right\|_K\leq r \right\}.
\]
Also, in these same coordinates, $\pa G$ can be locally parametrized by a convex function $\psi\in C^2$
and, clearly, $\psi(0)=0$ and $\nabla \psi(0)=0.$ Furthermore, our choice of the axes $e_1,\ldots,e_N$ 
allow us to write that 
\[
\psi (z')=\frac{1}{2}\sum_{i=1}^{N-1}\ka_i(x)z_i^2+o(\|z'\|^2),
\]
for $z'=(z_1,\ldots,z_{N-1})\in\mathbb R^{N-1}$ in a sufficiently small neighborhood of $0.$
\par
We need to estimate the measure of the remainder set
$$
R(x,r)=B_K(x,r)\cap H_{\nu(x)} ^+\setminus G;
$$
for sufficiently small $r>0,$ $R(x,r)$ can be written as 
\[
\left\{ x+\sum_{i=1}^{N-1}z_i e_i+z_N \nu : 
\left\|\sum_{i=1}^{N-1}z_i e_i +z_N \nu\right\|_K\leq r,\,0\leq z_N\leq \psi(z'), z'\in V \right\},
\]
where $V$ is some neighborhood of $0$ in $\RE^{N-1}.$
Next, we make the following change of variables: $z_i=r\xi_i$, for $i=1,\ldots,N-1$ and $z_N=r^2\xi_N;$ since $\|\cdot\|_K$ is positively homogeneous, we get that
\[
V(R(x,r))=r^{N+1}V(S_r),
\]
where $S_r$ is the set
\[
\left\{\xi\in\RN:\,\xi'\in r^{-1}V,\,\left\|\sum_{i=1}^{N-1}\xi_i e_i+r\xi_N \nu\right\|_K\leq 1;\,0\leq\xi_N\leq \frac{\psi(r\xi_1,\ldots,r\xi_{N-1})}{r^2}\right\}.
\]
Now, if we define the set
\[
S_0 =\left\{\xi\in\RN : \left\|\sum_{i=1}^{N-1}\xi_i e_i(x)\right\|_K< 1;0\leq\xi_N< \frac{1}{2}\sum_{i=1}^{N-1}\ka_i(x)\,\xi_i^2\right\},
\]
we easily check that
$$
S_0\subseteq \bigcup_{r>0}\Bigl(\bigcap_{0<t<r} S_t\Bigr)\subseteq 
\bigcap_{r>0}\Bigl(\bigcup_{0<t<r} S_t\Bigl)\subseteq\ovr S_0.
$$
Since $V(S_0)=V(\ovr S_0),$ the smoothness assumptions on $\pa G$ give
the sufficient uniform boundedness to infer that
$$
    \lim_{r\to 0^+}\frac{V(R(x,r))}{r^{N+1}}=V(S_0).
$$
By the definition of $S_0,$ $V(S_0)$ is easily computed as
\[
V(S_0)=\int_{K\cap \pi_{\nu(x)}}\frac{1}{2}\sum_{i=1}^{N-1}\ka_i(x)\,\xi_i^2\, d\xi
     =\frac{1}{2}\sum_{i=1}^{N-1}m_i(x)\,\ka_i(x),
\]
that implies the desired formula \eqref{curvatures}.
\end{proof}

\begin{cor}\label{curvaturecor}
Let $G$ be a $C^2$-smooth $K$-dense body.
Then there exists a positive constant $\al$ such that \eqref{cond2} holds.
\end{cor}

\begin{proof}
That the right-hand side of \eqref{cond2} does not depend on $x$ for $x\in\pa G$
clearly follows from \eqref{cdense} and \eqref{taylor}. 
Since $K$ is a convex body, then the $m_i(x)$'s are all positive;
if $\al$ were zero, then all the curvatures would be zero for every $x\in\pa G$
and this is impossible, since $G$ is a convex body. 
\end{proof}

\setcounter{equation}{0}

\section{Alternative proof of the conjecture in the two-dimensional case}

In this section, we present our new proof of the result of Amar, Berrone and Gianni \cite{ABG}.
We stress the fact that, besides dropping the smoothness assumptions needed in \cite{ABG},
our proof only needs the pointwise information given by \eqref{cond2}, since it 
relies on some global information provided by the Minkowski and affine isoperimetric
inequalities. So far, we were not able to reproduce this proof in general dimension.
\par
We need to introduce some terms and notations that we borrow from the theory of 
convex bodies (see \cite{Sc} and \cite{Lu2}, for instance). 
We limit our presentation to the case $N=2.$
\par
Given a convex body $K,$ we denote by $\rho_K$ and $h_K$ its {\em radial function} and
{\em support function,} respectively; by our notations, we have that $\rho_K(u)=1/\|u\|_K$ for $u\in\S1.$
The only moment of inertia $m=m_1$ in \eqref{defmi} can be easily computed and,
by setting $u=\nu(x),$ re-defined as a function on $\S1$ as 
\begin{equation}
\label{m2D}
m(u)=\frac23\,\rho_K(u^\perp)^3, \ u\in\S1, 
\end{equation}
where $u^\perp$ is the unit vector obtained from $u$ by a clockwise rotation of $90$ degrees.
\par
The {\em curvature function} $f_K$ of $K$ can be defined as a non-negative function on $\S1$ 
such that the {\em mixed volume} $V(K,G)$ can be written as
\begin{equation}
\label{mixed}
V(K,G)=\frac12\,\int_\S1 f_K(u)\,h_G(u)\,du,
\end{equation}
for every compact convex set $G.$
When $K$ is smooth, $f_K(u)$ is the reciprocal of the curvature $\ka_K$ of $\pa K$
at the point on $\pa K$ at which the normal unit vector equals $u.$
The Minkowski's first inequality for mixed volumes tells us that
\begin{equation}
\label{AF}
V(K,G)\ge \sqrt{V(K)\,V(G)};
\end{equation}
the sign of equality holds if and only if $K$ and $G$ are homothetic.
\par
We recall that the {\em affine area} $\Om(K)$ of $K$ is defined by
\begin{equation}
\label{affine}
\Om(K)=\int_\S1 f_K(u)^{2/3}\, du;
\end{equation}
we will make use of an inequality, that relates $\Om(K),$
$V(K)$ and the volume of the polar set $K^*$ (with respect to the origin) of $K$ 
and can be found in \cite{H} or \cite{Lu3}:
\begin{equation}
\label{lutwak}
\Om(K)^3\le 8\, V(K)^2 V(K^*);
\end{equation}
here, the sign of equality holds if and only if there exists a positive constant $\lambda$ such that $f_K(u)=\lambda h_K(u)^{-3}$, for all $x\in\partial K$.\\
In \cite{Pe} Petty proves that the latter condition holds if and only if $K$ is an ellipse.


\begin{thm}
\label{th:homothetic}
Let $K\subset\mathbb R^2$ be a convex body. If $G\subset\mathbb R^2$ is a $K$-dense body,
then $G$ and $K$ are homothetic and both $\pa K$ and $\pa G$ are $C^\infty$-smooth.
\end{thm}

\begin{proof}
Since $K$ is centrally symmetric by Corollary \ref{cor:2dreg},
then Corollary \ref{centroids} implies that formula \eqref{cond2} holds and, by \eqref{m2D} 
and in view of the geometric meaning of the curvature function, 
can be written as
\begin{equation}
\label{cond22d}
\rho_K(u^\perp)^3=c\,V(K)\,f_G(u), \ u\in\S1,
\end{equation}
(with a slight abuse of notation) where $c$ is some positive constant.
Also, being $K$ centrally symmetric, $\rho_K(-u^\perp)=\rho_K(u^\perp)$ 
and hence $f_G(-u)=f_G(u)$ for every $u\in\S1;$ this means that also $G$ is centrally symmetric.
\par
Thus, by Corollary \ref{cor:omegasim}, $K$ and $G$ differ by a homothety and 
both $\pa K$ and $\pa G$ are $C^\infty$-smooth.
\end{proof}


\begin{proof}[Proof of Theorem \ref{th:ellipse}.]
In view of Theorem \ref{th:homothetic}, we know that $G$ and $K$ have smooth 
boundaries and only differ by a homothety; without loss of generality, we shall 
assume that $G=K.$ Thus, \eqref{cond22d} reads:
\begin{equation}
\label{cond22df}
\rho_K(u^\perp)^3=c\,V(K)\,f_K(u), \ u\in\S1.
\end{equation}
\par
Our goal is to show that \eqref{cond22df} leads inequality \eqref{lutwak} into an equality; then we shall conclude that $K$ is an ellipse.
\par
By a well-known formula, we then compute:
\begin{eqnarray*}
&&2 V(K)=\int_\S1 \rho_K(u)^2 du=\int_\S1 \rho_K(u^\perp)^2 du=\\
&&\qquad\qquad [c\, V(K)]^{2/3} \int_\S1 f_K(u)^{2/3}\, du=
[c\, V(K)]^{2/3}\, \Om(K),
\end{eqnarray*}
that gives:
$$
c^{-2}=\frac{\Om(K)^3}{8 V(K)}.
$$
\par
On the other hand, by the definition \eqref{mixed}, \eqref{cond22df} also gives:
\begin{eqnarray*}
&&V(K, (K^*)^\perp)=\frac12\,\int_\S1 f_K(u)\,h_{K^*}(u^\perp)\,du=
\frac12\,\int_\S1 \frac{f_K(u)}{\rho_K(u^\perp)}\,du=\\
&&\qquad[c\, V(K)]^{-1}\,\frac12\,\int_\S1 \rho_K(u^\perp)^2\,du=
c^{-1}
\end{eqnarray*}
where we have used the well-known fact that $h_{K^*}=1/\rho_K.$ 
\par
Therefore, by applying \eqref{AF} and \eqref{lutwak} successively, we obtain that
\begin{eqnarray*}
&&\frac{\Om(K)^3}{8 V(K)}=c^{-2}=V(K, (K^*)^\perp)^2\ge V(K)\,V((K^*)^\perp)=\\
&&\qquad\qquad V(K)\,V(K^*)\ge \frac{\Om(K)^3}{8 V(K)}\,,
\end{eqnarray*}
that is Aleksandrov-Fenchel inequality holds with the sign of equality, which
means that $K$ and $(K^*)^\perp$ are homothetic. This concludes the proof.
\end{proof}

\end{document}